\documentclass[twoside,11pt]{article}

\usepackage{graphicx}

\newcommand{\qed}{\hfill\rule{2.1mm}{2.1mm}}

\newtheorem{theorem}{Theorem}[section]

\newtheorem{definition}{Definition}[section]

\newtheorem{lemma}{Lemma}[section]

\usepackage{amsmath, amssymb,amstext, mathtools} 
\usepackage{subfig}
\usepackage{booktabs}
\usepackage{epstopdf}
\usepackage{cite} 

\usepackage{algcompatible}
 \usepackage{algpseudocode}
 \usepackage{algorithm,algorithmicx}
 \usepackage{bbm}

 \usepackage{hyperref}
 \usepackage[dvipsnames]{xcolor}
 \hypersetup{colorlinks=true, allcolors=NavyBlue}

 \usepackage{float}

 \usepackage[top=1in, bottom=1in, left=1in, right=1in]{geometry} 

 \DeclareMathOperator{\Diag}{{Diag}}

 \DeclareMathOperator{\rank}{{rank}}

 \DeclareMathOperator{\tr}{{Tr}}
 \DeclareMathOperator{\st}{{s.t.}}

 \newcommand{\A}{\bs{A}}
 \newcommand{\E}{\mathbb{E}}

 \newcommand{\R}{\mathbf{R}}

 \newcommand{\bs}{\boldsymbol}

 \renewcommand{\d}{\bs{d}}
 \newcommand{\e}{\bs{e}}
 \newcommand{\I}{\bs{I}}

 \renewcommand{\v}{\bs{v}}
 
 \newcommand{\W}{\bs{W}}
 \newcommand{\x}{\bs{x}}
 \newcommand{\X}{\bs{X}}
 \newcommand{\y}{\bs{y}}
 \newcommand{\Y}{\bs{Y}}
 \newcommand{\z}{\bs{z}}

 \newcommand{\0}{\bs{0}}

 \newcommand{\mat}[1] {\ensuremath{ \left(\begin{array} #1 \end{array} \right)}} 
 \newcommand{\branchdef}[1] {\ensuremath{ \left \{\begin{array}{rl} #1 \end{array} \right. }} 
 \newcommand{\sbra}[1] {\ensuremath{ \left[ #1\right]}} 
 \newcommand{\rbra}[1]{\ensuremath{\left( #1 \right)}} 
 \newcommand{\bra}[1]{\ensuremath{\left\{ #1 \right\}}} 

 \newcommand{\ra}{\rightarrow}

\begin{document}
\allowdisplaybreaks

\title{Convex optimization for the densest subgraph and densest submatrix problems}


\author{Polina Bombina \thanks{Department of Mathematics, University of Alabama,
 Tuscaloosa, Alabama,
 \href{mailto:pbombina@crimson.ua.edu}{pbombina@crimson.ua.edu}}
 \and  Brendan Ames \thanks{University of Alabama,
  Tuscaloosa, Alabama,
    \href{mailto:bpames@ua.edu}{bpames@ua.edu}}
  }



\maketitle

\begin{abstract}
We consider the densest $k$-subgraph problem, which seeks to identify the $k$-node subgraph of a given input graph with maximum number of edges. This problem is well-known to be NP-hard, by reduction to the maximum clique problem. We propose a new convex relaxation for the densest $k$-subgraph problem, based on a nuclear norm relaxation of a low-rank plus sparse decomposition of the adjacency matrices of $k$-node subgraphs to partially address this intractability. We establish that the densest $k$-subgraph can be recovered with high probability from the optimal solution of this convex relaxation if the input graph is randomly sampled from a distribution of random graphs constructed to contain an especially dense $k$-node subgraph with high probability. Specifically, the relaxation is exact when the edges of the input graph are added independently at random, with edges within a particular $k$-node subgraph added with higher probability than other edges in the graph. We provide a sufficient condition on the size of this subgraph $k$ and the expected density under which the optimal solution of the proposed relaxation recovers this $k$-node subgraph with high probability.
Further, we propose a first-order method for solving this relaxation based on the alternating direction method of multipliers, and empirically confirm our predicted recovery thresholds using simulations involving randomly generated graphs, as well as graphs drawn from social and collaborative networks.
\end{abstract}

\section{Introduction}
\label{sec:intro}

We consider the \emph{densest $k$-subgraph problem:} given graph $G=(V,E)$, identify the
$k$-node subgraph of $G$ of maximum density, i.e., maximum average degree.
Equivalently, the problem reduces to finding the $k$-node subgraph of $G$ with maximum number of edges.
It is easy to see that the densest $k$-subgraph problem is NP-hard by reduction to the maximum clique problem, well-known to be NP-hard~\cite{karp1972reducibility}.
Indeed, if $G$ contains a clique of size $k$, it would induce the densest $k$-subgraph of $G$; any polynomial time algorithm
for densest $k$-subgraph would immediately provide a polynomial-time algorithm for maximum clique. Moreover, it has been shown by \cite{alon2011inapproximability,feige2002relations,khot2006ruling} that the densest $k$-subgraph problem does not admit polynomial-time approximation schemes in general.
Despite this intractability, the identification of dense subgraphs plays a significant role in many practical applications, especially in the analysis of web graphs, and social and biological
networks~\cite{henzinger2002challenges,  gibson2005discovering,angel2012dense,gajewar2012multi,tsourakakis2015k,tsourakakis2013denser}.

We propose a new convex relaxation for the densest $k$-subgraph problem to address this intractability.
Although we do not, and should not, expect our algorithm to provide a good approximation of the densest $k$-subgraph
for all graphs,
we will show that it is functionally equivalent to the densest $k$-subgraph problem for a large class of program instances.
In particular, suppose that the random  input graph consists of a $k$-node subgraph $H$ with edges added
with significantly higher probability than those edges outside $H$. We will show that if $k$ is sufficiently large then
$H$ is the densest $k$-subgraph of $G$, and it can be recovered from the optimal solution of our convex relaxation.

This result can be thought of as a specialization of recent developments regarding the recovery of clusters in graphs.
In graph clustering, one seeks to partition the nodes of a given graph into dense subgraphs. Several recent results \cite{abbe2016exact,ailon2013breaking,ames2014convex, ames2014guaranteed,amini2014semidefinite,
cai2015robust,chen2014clustering,chen2014improved,chen2014statistical,guedon2015community,hajek2015achieving,lei2015consistency,mathieu2010correlation,nellore2013recovery,oymak2011finding,rohe2011spectral,qin2013regularized, vinayak2014sharp}, among others, have established sufficient conditions on the generative model under which dense subgraphs can be recovered in a random graph, typically from the solution of some convex relaxation. These results assume that the random graph is generated using some generalization of the stochastic block model (see~\cite{holland1983stochastic}), which assumes that the edges are added within blocks or clusters with higher frequency than between blocks, and provide sufficient conditions on the number and relative sizes of clusters, and the probabilities of adding edges that guarantee that the underlying block structure can be recovered in polynomial-time. The recent survey article~\cite{li2018convex} provides an overview of such recovery guarantees.

Relatively few analogous results exist for the densest $k$-subgraph problem.
Ames and Vavasis~\cite{ames2011nuclear, ames2015guaranteed} consider convex relaxations for the maximum clique problem. Given an input graph, the \emph{maximum clique problem} aims to identify the largest clique in the graph, that is, the  vertex set of the largest complete subgraph (see \cite{bomze1999maximum,pardalos1994maximum} for further discussion of the maximum clique problem). Ames and Vavasis \cite{ames2011nuclear,ames2015guaranteed} establish that the maximum clique can be recovered from the optimal solution of particular convex relaxation if the input graph consists of a single large complete graph that is obscured by noise in the form of random edge additions and deletions. In particular, both results show that hidden cliques of size at least $\Omega(\sqrt{n})$ can be identified with high probability for $n$-node random graphs constructed so that the probability of adding an edge between nodes in the hidden clique is significantly higher than adding other potential edges to the graph.
The notation $f(x) = \Omega(g(x))$ indicates that there is some constant $C$ such that $f(x) \ge C g(x)$ for all sufficiently large $x$ and we say that an event occurs \emph{with high probability} if the probability of the event tends to $1$ polynomially as the size of the graph $N$ tends to $\infty$.
The latter result recasts the hidden clique problem as that of finding the densest $k$-subgraph, where $k$ is the size of the hidden clique. Similar theoretical recovery guarantees can be found in~\cite{deshpande2015finding, montanari2015finding, hajek2017information,hajek2016semidefinite,saad2018belief,hajek2018recovering}.
We delay the derivation of our convex relaxation and statement of the general recovery guarantee until Section~\ref{sec:dks}.

We generalize these results for the  densest subgraph problem to obtain an analogous recovery guarantee for the \emph{densest submatrix problem}, which seeks to find the densest submatrix of given size. That is, we seek the submatrix of desired size with maximum number of nonzero entries. Similar results for the submatrix localization problem, where one seeks to find a block of entries with elevated mean in a random matrix, were presented in~\cite{hajek2017submatrix, banks2018information, brennan2019universality}.
We will see that our convex relaxation correctly identifies the densest submatrix (of fixed size) in random matrices provided that entries within this submatrix are significantly more likely to be nonzero than an arbitrary entry of the matrix. We present our generalization of the densest subgraph problem to the densest submatrix problem and the statement of our theoretical recovery guarantees in Section~~\ref{sec:DSM}. We provide proofs of our main results in Section~\ref{sec:DSM-proof} and conclude with discussion of a first-order method for solving our convex relaxations and empirical results illustrating efficacy of our approach in Section~\ref{sec:expts}.

\section{Relaxation of the Densest $k$-Subgraph Problem and Perfect Recovery of a Planted Clique}
\label{sec:dks}

Our relaxation hinges on the observation, made in~\cite{ames2015guaranteed}, that the  adjacency matrix
of \emph{any} subgraph of $G$ can represented as the difference of a rank-one matrix and a  binary correction
matrix; this observation is closely related to the sparse plus low-rank decomposition of clustered graphs
first considered in~\cite{oymak2011finding, chen2012clustering}, although with the restriction to submatrices of the adjacency matrix.
Let $\hat V \subseteq V$ be a subset of nodes of $G = (V,E)$. We denote by $G(\hat V)$ the subgraph induced by $\hat V$; that is,
$G(\hat V)$ is the graph with node set $\hat V$ and edge set given by the subset of $E$ with both endpoints in $\hat V$.
Let $\v \in \R^V$ be the characteristic vector of $\hat V$:
$v_i = 1$ if $i\in \hat{V}$ and $v_i = 0$ otherwise for all $i \in V$.
The matrix $\hat \X = \v\v^T$ is a rank-one binary matrix with nonzero entries indexed by $\hat V \times \hat V$.
If $G(\hat V)$ is a complete subgraph, i.e., $ij \in E$ for all $i,j \in \hat V$,
then $\hat \X$ is equal to the sum of the adjacency matrix of $G(\hat V)$ and the binary diagonal matrix $I_{\hat V}$
with nonzero entries indexed by $\hat V$;
we call this sum the perturbed adjacency matrix of $G(\hat V)$, and denote it by $\bs {\tilde A_{G(\hat V)}}$.

If $G(\hat V)$ is not a complete subgraph, then there is some $(i,j) \in V \times V$, $i\neq j$
such that $ij \notin E$.
Let $\Omega$ denote the set of all such $(i,j)$. For each $(i,j) \in \Omega$, we have $\hat X_{ij} =1$, while ${[ \bs {\tilde A_{G(\hat V)}}]}_{ij} = 0$.
It follows that
$
	 \bs {\tilde A_{G(\hat V)}} = \hat \X - P_{\Omega}(\hat \X),
$
where $P_{\Omega}$ is the projection onto the set of matrices having support contained in $\Omega$, defined by
$$
	[P_\Omega(\bs M)] = \branchdef{ M_{ij},& \mbox{if  } (i,j) \in \Omega \\ 0, &\mbox{otherwise.}}
$$
We call $(\hat\X, \hat\Y) := (\hat\X, P_\Omega(\hat\X))$ the \emph{matrix representation} of the subgraph $G(\hat V)$.
The density of $G(\hat V)$ is given by
$$
	d(G(\hat V)) = \frac{1}{k} \rbra{ {k\choose 2} - \frac{1}{2} \| P_\Omega(\hat\X)\|_0 },
$$
where $\|\bs M \|_0$ denotes the number of nonzero entries of $\bs M$, because $\|P_\Omega(\hat\X)\|_0$ is equal to twice the number of
nonadjacent nodes in $\hat V$.
Moreover, the entries of the correction matrix $P_\Omega(\hat \X)$ are binary, which implies that $\|P_\Omega(\hat \X)\|_0 = \sum_{i,j\in V} {[P_\Omega(\hat\X)]}_{ij}$.
Therefore, we may pose the densest $k$-subgraph problem
as  the rank-constrained binary program
\begin{equation}\label{eq: dks 1}
  \begin{array}{ll}
    \min & \tr(\Y \e\e^T)  \\
    \st & \tr(\X \e\e^T) \ge k^2, \; P_\Omega(\X - \Y) = \0,\;  \rank(\X) = 1 \\
      & \X = \X^T, \; \Y = \Y^T,\; \X \in {\{0,1\}}^{V\times V}, \; \Y \in{ {\{0,1\}}^{V\times V}}.
  \end{array}
\end{equation}
where $\tr: \R^{n\times n} \ra \R^n$ denotes the matrix trace function, and $\e$ denotes the all-ones vector in $\R^V$.
Unfortunately, combinatorial optimization problems involving rank and binary constraints are intractable in general.
In particular, the densest $k$-subgraph problem is NP-hard and, hence, we cannot expect to be able to solve~\eqref{eq: dks 1} efficiently.
Relaxing the rank constraint with a nuclear norm penalty term given by
$\|\X \|_* = \sum_{i=1}^N \sigma_i(\X)$ as in~\cite{recht2010guaranteed},
the binary constraints with box constraints,
and ignoring the symmetry constraints yields
the convex program
\begin{equation}\label{eq: dks rel}
  \begin{array}{ll}
    \min & \|\X \|_* + \gamma \tr(\Y \e\e^T) \\
  	\st & \tr(\X \e\e^T) = k^2, \;P_\Omega(\X - \Y) = \0,\;\0 \le \X \le \e\e^T,\; \0 \le \Y, 
  \end{array}
\end{equation}
where $\gamma > 0$ is a regularization parameter chosen to control emphasis between
the two objectives.

As mentioned earlier, we do not expect the solution of~\eqref{eq: dks rel}
to give a good approximation of the densest $k$-subgraph for an arbitrary graph.
We instead restrict our focus to those graphs which we can expect to to contain a single especially dense $k$-subgraph with high probability.

\begin{definition}
	We construct the edge set of a $N$-node random graph $G = (V,E)$ as follows.
	Let $V^* \subseteq V$ be a $k$-subset of nodes;
	for each $(i,j) \in V^* \times V^*$, we add $ij$ to $E$ independently with probability $q$.
	For each $(i,j) \in (V\times V) - (V^*\times V^*)$, we add $ij$ to $E$ independently
	with probability $p < q$.
	We say such a graph $G$ is sampled from the
  \emph{planted dense $k$-subgraph model}.
\end{definition}

This model, considered in~\cite{ames2015guaranteed}, is a generalization of the planted clique model considered in~\cite{ames2011nuclear}, where $q$ is chosen to be $q = 1$.
On the other hand, the planted dense $k$-subgraph model is a special case of the generalized stochastic block model~\cite{chen2012clustering},
corresponding to a graph with exactly one cluster of size $k$ and $N-k$ outlier nodes.
Note that any graph $G$ sampled from the planted dense $k$-subgraph contains a $k$-subgraph, $G(V^*)$, with higher density  than the rest of the graph in expectation.
Our goal is to derive conditions on the size $k$ of this subgraph and the edge densities $p$ and $q$ that ensure recovery of the planted subgraph $G(V^*)$ from the optimal solution of~\eqref{eq: dks rel}. The following theorem provides such a sufficient condition.

\begin{theorem}\label{thm: rec}
	Suppose that the $N$-node graph $G = (V,E)$ is sampled from the planted dense $k$-subgraph model
  with edge probabilities $q$ and $p$ respectively.
  Let $(\X^*, \Y^*)$ denote the matrix representation of the planted dense $k$-subgraph $G(V^*)$.
  Then constants $c_1, c_2, c_3> 0$ exist such that if
  \begin{equation}\label{eq: gap 1}
    q-p \ge c_1 \max \bra{ \sqrt{ \max\{\sigma_q^2, \sigma_p^2\} \frac{ \log N}{k} },  \frac{\log N}{k}\sqrt{ \sigma_p^2 N} , \frac{(\log N)^{3/2}}{k} }
  \end{equation}
  then $(\X^*, \Y^*)$ is the unique optimal solution of~\eqref{eq: dks rel}
  for penalty parameter
  \begin{equation} \label{eq: gamma choice}
			\gamma \in \rbra{ \frac{c_2}{(q-p)k}, \frac{c_3}{(q-p)k} },
		\end{equation}
	and $G(V^*)$ is the unique densest $k$-subgraph of $G$ with high probability; here $\sigma_q^2$ and $\sigma_p^2$ are equal to
  the edge creation variances $q(1-q)$ and $p(1-p)$ inside and outside of the planted dense $k$-subgraph, respectively.
\end{theorem}

Here, and in the rest of the paper, an event holding \emph{with high probability (w.h.p.)}, means
that the event occurs with probability tending polynomially to one as $N \ra \infty$;
that is, there are scalars $\hat c_1, \hat c_2 > 0$ such that the event occurs with probability at least
$1 - \hat c_1 N^{-\hat c_2}$.
Note that~\eqref{eq: gap 1} is only satisfiable when $k = \Omega((\log N)^{3/2})$.
To illustrate the contribution of Theorem~\ref{thm: rec} we consider a few choices of $p,q$, and $k$.

First, suppose that $p$ and $q$ are fixed so that the edge densities in $G$ are fixed as we vary $N$.
In this case, Theorem~\ref{thm: rec} states that we may recover $G(V^*)$, with high probability, provided that
$ k = \Omega\rbra{(\log N)^{3/2}}$.
This bound is identical to that found many times in the planted clique literature~\cite{ames2011nuclear,alon1998finding,feige2010finding,dekel2014finding,deshpande2015finding},
up to constants and the logarithmic term.
It is widely believed that finding planted cliques of size $o(\sqrt{N})$ is intractable;
indeed, several heuristic approaches have recently been proven to fail to recover planted cliques of size $o(\sqrt{N})$ in
polynomial-time~\cite{feige2002relations,alon2011inapproximability,khot2006ruling} and this intractability has been exploited in cryptographic applications~\cite{juels2000hiding}.

On the other hand, our bound shows that planted cliques of size much smaller than $\sqrt{N}$
can be recovered in the presence of \emph{sparse} noise. This should not be seen as a proof that we can recover planted cliques of size $o(\sqrt{N})$ in general, but rather evidence of the intimate relationship between the size of hidden cliques recoverable and the noise obscuring them. If very little noise in the form of diversionary edges is hiding the signal, here the planted clique, we should expect
the signal to be significantly easier to recover. This is reflected in the fact that we can recover significantly smaller cliques than $o(\sqrt{N})$ in this setting.
For example, let $q$ be a fixed constant and let $p$ vary with $N$ such that $p  \le \log N/N$.
The probability of adding an edge outside of $G(V^*)$ tends to zero as $k, N \ra \infty$.
Further, the left-hand side of~\eqref{eq: gap 1} tends to $q$ as $N \ra \infty$, and the dominant term in  the right-hand side  is $(\log N)^{3/2} / k $.
This implies that we can have exact recovery  of the hidden clique w.h.p.~provided $k = \Omega\rbra{(\log N)^{3/2}}$.
This lower bound on the size of recoverable $k$-subgraph matches that for identifying clusters
in sparse graphs provided in the graph clustering literature,
albeit for the case where the graph contains the single cluster indexed by $V^*$ surrounding by many outlier nodes (see~\cite{li2018convex} and the references within).
Moreover, this lower bound improves significantly upon that given by~\cite{ames2015guaranteed}, where it is shown to that $k=\Omega(N^{1/3})$ is sufficient for exact recovery w.h.p.~in the presence of sparse noise.

\section{The Densest Submatrix Problem}
\label{sec:DSM}

The densest $k$-subgraph problem is a specialization of the far more general densest submatrix problem.
Let ${[M]} = \{1,2,\dots, M\}$ for each positive integer $M$.
Given a matrix $\A \in \R^{M\times N}$, the \emph{densest $m\times n$-submatrix problem} seeks subsets $\bar U \subseteq {[M]}$ and $\bar V \subseteq {[N]}$ of cardinality
$|\bar U|=m$ and $|\bar V| = n$, respectively,
such that the submatrix $\A{[\bar U, \bar V]}$ with rows index by $\bar U$ and columns indexed by $\bar V$
contains the maximum number of nonzero entries.
It should be clear that this specializes immediately to the densest $k$-subgraph problem when the input matrix  is the perturbed adjacency matrix $\A= \A_G + \I$ of the input graph and $m= n=k$.
However, the densest $m\times n$-submatrix problem allows far more flexible problem settings.

For example, the densest submatrix problem also specializes immediately to the maximum edge/density biclique problem. Let $G = (U,V, E)$ be a bipartite graph. Given integer $m,n$, the decision version of the maximum edge biclique problem
determines if $G$ contains an $m\times n$ biclique, i.e., whether there are vertex sets $\bar U \subseteq U$, $\bar V \subseteq V$ of cardinality $|\bar U|=m$ and $|\bar V| = n$ such that each vertex in $\bar U$ is adjacent to every vertex in $\bar V$.
This problem immediately specializes to the densest $m\times n$-subgraph problem with $\A$ equal to the $(U,V)$-block of the adjacency matrix of $G$.
Similar specializations exist for finding the densest subgraph in directed graphs, hypergraphs, and so on.

Let $\Omega$ denote the index set of zero entries of given matrix $\A \in \R^{M\times N}$. Without loss of generality, we may assume that the entries of $\A$ are binary.
If not, then we may replace $\A$ with the binary matrix having the same sparsity pattern without changing the index set of the densest $m\times n$-submatrix.
We would like to obtain a rank-one matrix $\X$ with  $mn$ nonzero entries with minimum number of disagreements $\A$ on $\Omega$:
\begin{equation}\label{eq: dsm prob}
  \begin{array}{cl}
    \displaystyle \min_{\X, \Y \in { \{0,1\}}^{M\times N} } &  \tr(\Y\e\e^T) \\
    \st & \tr(\X \e\e^T) = mn, \;P_\Omega(\X - \Y) = \0,\; \rank \X = 1,
  \end{array}
\end{equation}
where $\Y$ is used to count the number of disagreements between $\A$ and $\X$.
Relaxing binary and rank constraints as before, we obtain the convex relaxation
\begin{equation}\label{eq: dsm rel}
\begin{array}{cl}
	\displaystyle\min & \|\X \|_* + \gamma \tr(\Y \e\e^T)  \\
	\st &  \tr(\X \e\e^T) = mn, \;P_\Omega(\X - \Y) = \0, \; \0 \le \X \le \e\e^T,\; \0 \le \Y,
\end{array}
\end{equation}
where $\gamma >0$ is a regularization parameter chosen to tune between the two objectives.
As before, we should expect to recover the solution of~\eqref{eq: dsm prob} from that of~\eqref{eq: dsm rel} when $\A$ contains
a single large dense $m\times n$ block. The following definition proposes a class of random matrices with this property.

\begin{definition}
	We construct an  $M\times N$ random binary matrix $\A$ as follows.
	Let $U^* \subseteq {[M]}$ and $V^* \subseteq {[N]}$ be $m$ and $n$-index sets.
	For each  $i \in U^*$ and $j\in V^*$, we let $a_{ij} = 1$ with probability $q$ and
  $0$ otherwise.
	For each remaining $(i,j)$ we set $a_{ij} = 1$ with probability $p < q$ and take $a_{ij} = 0$ otherwise.
	We say such a matrix  $\A$ is sampled from the \emph{planted dense $m\times n$-submatrix model.}
\end{definition}

The following theorem provides a  sufficient condition for exact recovery of a planted dense $m\times n$-submatrix generalizing
the analogous result for recovery of a planted dense $k$-subgraph given by Theorem~\ref{thm: rec}.

\begin{theorem}\label{thm: dsm rec}
  Suppose that the matrix $\A \in \R^{M\times N}$ is sampled from the planted dense $m\times n$-subgraph model
  with edge probabilities $q$ and $p$, respectively, with rows and columns of the planted dense subgraph
  indexed by $U^*$ and $V^*$, respectively.
  Let $(\X^*, \Y^*)$ denote the matrix representation of  $\A(U^*, V^*)$.
  Let $N_{\max} := \max\{M,N\}$ and $n_{\min} := \min\{m,n\}$.
  Then there are constants $c_1, c_2, c_3> 0$ such that if
  \begin{equation}\label{eq: dsm gap}
    q-p \ge c_1 \max \bra{ \sqrt{\max\{\sigma^2_q, \sigma^2_p\} \frac{ \log N_{\max} }{n_{\min}} },
    \frac{\log N_{\max}}{n_{\min}} \sqrt{ \sigma_p^2 N_{\max}},
     \frac{(\log N_{\max})^{3/2} }{ n_{\min}} }
  \end{equation}
  then $(\X^*, \Y^*)$ is the unique optimal solution of~\eqref{eq: dsm rel}
  for penalty parameter
  $ \gamma = t/((q-p)n_{\min})$
    for all $c_2 \le t \le c_3$,
	and $\A(U^*, V^*)$ is the unique densest $m\times n$-submatrix of $\A$ with high probability.
\end{theorem}

In the case when $M=N$ and $m=n$, the inequality~\eqref{eq: dsm gap} specializes to~\eqref{eq: gap 1}, although the constants $c_1, c_2, c_3$ should differ due to the lack of an assumption of symmetry of $\X^*$ and $\Y^*$ in Theorem~\ref{thm: dsm rec}.


\section{Derivation of the Recovery Guarantees}
\label{sec:DSM-proof}

This section will consist of a proof of Theorem~\ref{thm: dsm rec}.
 The proof of~Theorem~\ref{thm: rec} is identical except for minor modifications due to the symmetry of $\A$.
%
%
%
We begin with the following theorem, which provides the required optimality conditions for~\eqref{eq: dsm rel}.
\renewcommand{\u}{{\bs{u}}}

\begin{theorem}\label{thm: KKT}
	Let $\bar U \subseteq \{1,\dots, M\}$ be a $m$-subset of $[M]$
	and let $\bar V \subseteq \{1, \dots, N\}$  be a $n$-subset of $[N]$, and
	$\bar\u, \bar\v$ be their characteristic vectors.
	Then the solutions $\bar\X = \bar\u \bar \v^T$ and $\bar \Y = P_\Omega(\bar\X)$
	are optimal  for~\eqref{eq: dsm rel} if and only
	if there are dual multipliers $\lambda \ge 0$,  $\bs\Lambda \in \R^{M\times N}_+$, $\bs\Xi \in \R^{M\times N}_+$,
	and $\W \in \R^{M\times N}$ satisfying
	\begin{subequations}\label{eq: KKT}
	\begin{align}
		\frac{\bar\u\bar \v^T}{\sqrt{mn}} + \W - \lambda \e\e^T  +\gamma \e\e^T- \bs\Xi + \bs\Lambda &= \0 	\label{eq: DF} \\
		\tr(\bs\Lambda^T(\bar\X - \e\e^T)) &= 0 \label{eq: X CS} \\
		\tr(\bs\Xi^T \bar\Y) &= 0 \label{eq: Y CS} \\
		\W^T \bar\u = \0, \; \W \bar \v =\0,\; \|\W \| & \le 1 \label{eq: NNM subdiff}.
	\end{align}
	\end{subequations}
\end{theorem}

The proof of~Theorem~\ref{thm: KKT} is nearly identical to that of \cite[Theorem~4.1]{ames2015guaranteed} and is omitted.
Suppose that $\A$ is sampled from the planted dense $(m,n)$-subgraph model
with edge probabilities $q > p$. Our goal is to establish the
conditions on $m,n, q, p$ given by Theorem~\ref{thm: dsm rec} that guarantee
exact recovery (w.h.p.) of the matrix representation $(\bar \X, \bar \Y)$ of
the planted submatrix with rows and columns given by $\bar U$ and $\bar V$ respectively.
Our approach follows that of~\cite[Section 4]{ames2015guaranteed}.
We first explicitly construct dual multipliers $\W$ and $\bs\Xi$ using the dual feasibility condition~\eqref{eq: DF}
and the complementary slackness conditions~\eqref{eq: X CS} and~\eqref{eq: Y CS}.
We then use the characterization of the subdifferential of the nuclear norm~\eqref{eq: NNM subdiff}
to construct the remaining dual variables $\lambda, \bs \Lambda$.
We conclude the proof by using concentration inequalities to  establish feasibility of the proposed dual variables under the hypothesis of Theorem~\ref{thm: rec}.

We choose $\W$ and $\bs \Xi$ according to the dual feasibility condition~\eqref{eq: DF} so that the orthogonality conditions $\W\bar\v = \0$ and $\W^T\bar\u = \0$
are satisfied.
We consider the following cases.

\begin{itemize}
	\item[]\label{cs: W1}
		{\bf Case 1.}
		If $(i,j) \in \bar U \times \bar V - \Omega$, then~\eqref{eq: DF} implies that
		$$
			W_{ij} = \lambda - \frac{1}{\sqrt{mn}} - \Lambda_{ij} =: \tilde \lambda - \Lambda_{ij},
		$$
		if we take $\Xi_{ij} = \gamma$ and define $\tilde \lambda := \lambda - 1/\sqrt{mn}.$
	\item[]\label{cs: W2}
 		{\bf Case 2.}
 		If $i \in \bar U$, $j \in \bar V$, and $(i,j) \in \Omega$, then we have
		$\bar X_{ij} =  \bar Y_{ij} = 1/\sqrt{mn}$, so $\Xi_{ij} = 0$ by~\eqref{eq: Y CS}.
		It follows that $W_{ij} = \tilde\lambda - \gamma - \Lambda_{ij}$ in this case.
	\item[]\label{cs: W3}
		{\bf Case 3.}
		If $(i,j) \notin \bar U \times \bar V$ such that $(i,j) \notin \Omega$ then we take
		$ W_{ij} = \lambda$ and $\Xi_{ij} = \gamma$.
    \item[]\label{cs: W4}
      {\bf Case 4.}
      If $i \notin \bar U$, $j \notin \bar V$ such that $(i,j) \in \Omega$, we take
      $ W_{ij} = -\lambda p/(1-p)$ and $\Xi_{ij} =  \gamma - \lambda/(1-p)$.
    \item[]\label{cs: W5}
      {\bf Case 5.}
      If $i \in  \bar U$ and $j \notin \bar V$ such that $ij \in \Omega$, we take
      $$
        W_{ij} =  - \lambda \rbra{\frac{\nu_j}{m - \nu_j}},
      $$
      where $\nu_j$ denotes the number of nonzero entries in the
      $j$th column of $\A$ indexed by rows in $\bar U$.
      so that ${[\W^T \bar \u]}_j = 0$. By our choice of $W_{ij}$, we have
      $$
        \Xi_{ij} = \gamma - \frac{\lambda m}{m - \nu_j}.
      $$
    \item[]\label{cs: W6}
      {\bf Case 6.}
      If $i \notin \bar U$, $j\in \bar V$, and $(i,j) \in \Omega$ then we take
      $$
        W_{ij} = - \frac{\lambda \mu_i}{n-\mu_i} \hspace{0.5in}
        \Xi_{ij} = \gamma - \frac{\lambda n}{n -\mu_i},
      $$
      where $\mu_i$ denotes the number nonzero entries in the $i$th row
      of $\A$ indexed by columns in $\bar V$.

  \end{itemize}

  By our choice of $\W$ and $\bs\Xi$, we have ${[\W\bar\v]}_i =0$  for all $i \notin \bar U$ and ${[\W^T\bar \u]}_i = 0$ for all $i \notin \bar V$.
  We choose the remaining dual variables $\lambda$ and $\bs\Lambda$ so that ${[\W\bar \v]}_i = 0$ for all $i \in \bar U$ and ${[\W^T\bar \u]}_i =0$ for all $i \in \bar V$.
  %
  \newcommand{\bm}{{\bs{\bar\mu}}}
  \newcommand{\bn}{{\bs{\bar\nu}}}
  The  orthogonality conditions  $\W^T\bar \u = \bs 0$ and $\W\bar\v = \bs 0$
  define
  a linear system with $m+n$ equations for the $mn$ unknown entries of $\bs\Lambda$
  when all other dual variables are fixed.
  To obtain a particular solution of this underdetermined linear system,
  we make the additional assumption that $\bs\Lambda(\bar U, \bar V)$ has rank at most 2, taking the form
  $
    \bs\Lambda(\bar U, \bar V) = \y \e^T + \e\z^T
  $
  for some  $\y \in \R^m$ and $\z \in \R^n$.
  Under this assumption, the conditions ${[\W\v]}_i =0$, $i\in \bar U$
  and  ${[\W^T\u]}_j = 0$, $j \in \bar V$
  yield the linear system
  \begin{equation} \label{eq: yz sys}
    \mat{{cc} n \bs I & \e\e^T \\ \e \e^T & m \bs I }
    \mat{{c} \y \\ \z }
    =
    \mat{{c} - \gamma \bs{ \bar\mu} + n \tilde \lambda \e \\
      - \gamma \bs {\bar \nu} + m \tilde \lambda \e },
  \end{equation}
  where the vectors $\bm$ and $\bn$ are defined by
  $\bar\mu_i = n - \mu_i$ for all $i \in \bar U$
  and $\bar \nu_j = m - \nu_j$ for all $j \in \bar V$.
  It is easy to see that this system is singular with null space spanned by $(\e; -\e)$.
  However, it is also easy to see that the unique solution of
  \begin{equation} \label{eq: yz sys2}
    \mat{{cc} n \bs I + \e\e^T & \bs 0 \\
        \bs 0 & m \bs I + \e\e^T }
    \mat{{c} \y \\ \z }
    =
    \mat{{c} - \gamma \bs{ \bar\mu} + n \tilde \lambda \e \\
      - \gamma \bs {\bar \nu} + m \tilde \lambda \e }
  \end{equation}
  is  a solution of~\eqref{eq: yz sys};
  see~\cite[Section 4.2]{ames2015guaranteed} for further details.
  Applying the Sherman-Morrison-Woodbury Formula (see~\cite[Equation (2.1.4)]{GV}), we have
  \begin{equation} \label{eq: y z formula}
    \y =  \frac{1}{n} \rbra{ \tilde \lambda \frac{n^2}{m+n}
        - \gamma \bm + \gamma \frac{\bm^T \e}{m+n} \e }, \;
    \z = \frac{1}{m} \rbra{ \tilde \lambda \frac{m^2}{m+n} - \gamma \bn
        + \gamma \frac{\bn^T \e}{m+n} \e }. 
  \end{equation}
  The entries of $\bm$ and $\bn$
  are binomial random variables corresponding to $n$ and $m$ independent Bernoulli trials ith probability of success $1-q$, respectively.
  Therefore, we have
  \begin{equation} \label{eq: Ey Ez}
    \E[\y] = \frac{n}{m+n} \rbra{ \tilde\lambda - \gamma (1-q) } \e
    \hspace{0.25in}
    \E[\z] = \frac{m}{m+n} \rbra{ \tilde\lambda - \gamma (1-q) } \e.
  \end{equation}
  Choosing
  $
    \lambda = \frac{1}{\sqrt{mn}} + \gamma (1-q) + \gamma \tau
  $
  for some $\tau >0$ to be chosen later ensures that the entries of $\bs\Lambda$ are strictly positive in expectation.

  We next describe how to choose $\tau$ so that the entries of $\y$ and $\z$ are positive with high probability.
  To do so, we will make repeated use of the following specialization of the classical Bernstein inequality to bound the sum of independent Bernoulli  random variables
  (see, for example,\cite[Section~2.8]{lugosi2009}).

  \begin{lemma}\label{lem: Bernstein}
    Let $x_1, \dots, x_k$ be a sequence of $k$ independent $\{0,1\}$ Bernoulli random variables,
    each with probability of success $\rho$.
    Let $s = \sum_{i=1}^k x_i$ be the binomially distributed random variable denoting the number of successes.
    Then
    \begin{equation} \label{eq: Bernstein}
       \Pr \rbra{ |s - \rho k| > 6 \max \bra{ \sqrt{\rho (1-\rho) k \log t}, \log t } }  \le 2 t^{-6} .
    \end{equation}
  \end{lemma}


  Applying~\eqref{eq: Bernstein} with $t = N$ to each component of $\bm$ and $\bn$  and the union bound shows that
  \begin{align}\label{eq: mi}
    |\bar \mu_i - (1-q) n |  & \le 6 \max \{ \sqrt{\sigma_q^2 n \log N}, \log N \} \\
    |\bar \nu_j -  (1-q) m | & \le 6 \max \{ \sqrt{\sigma_q^2 m \log N}, \log N \}
    \label{eq: ni}
  \end{align}
  for all $i \in \bar U$  and $j \in \bar V$ w.h.p.,
  where $\sigma_q^2 = q(1-q)$.
   On the other hand, $\bn^T\e  =  \bm^T \e$ is equal
  to the number of nonzero entries in the $\bar U \times \bar V$ block of $\A$.
  Therefore, $\bn^T\e  =  \bm^T \e$ is a binomially distributed
  random variable, with $\E[\bn^T\e] = \E[\bm^T \e] = mn(1-q)$.
  Applying~\eqref{eq: Bernstein} with $t = N$ again establishes that
  \begin{equation} \label{eq: m sum}
    |\bn^T \e - (1-q)mn| = |\bm^T\e - (1-q)mn | \le 6 \max \{ \sqrt{\sigma_q^2 mn \log N}, \log N \}
  \end{equation}
  w.h.p.
  It follows immediately that
  \begin{align}
    | y_i &- \E[y_i] | \le \frac{\gamma}{n} \rbra{ | \bar \mu_i - \E[\bar\mu_i]|
      + \frac{1}{m+n} | \bm^T \e - \E{[\bm^T \e]} | } \notag \\
    &\le 6 \gamma \rbra{ 1 + \frac{1}{\sqrt{m}} } \max \bra{ \sqrt{\sigma_q^2 \frac{\log N}{n} }
    ,  \frac{\log N}{n} } \label{eq: y bound}
  \end{align}
  for each $i \in \bar U$ if~\eqref{eq: mi} and~\eqref{eq: m sum}
  are satisfied.
  Following an identical argument, we see that
  \begin{equation} \label{eq: z bound}
    |z_i - \E[z_i] | \le
    6 \gamma \rbra{ 1 + \frac{1}{\sqrt{n}} } \max \bra{ \sqrt{\sigma_q^2 \frac{\log N}{m} }
    ,  \frac{\log N}{m} }
  \end{equation}
  if~\eqref{eq: ni} and~\eqref{eq: m sum} hold.
  Substituting~\eqref{eq: y bound} and~\eqref{eq: z bound} into
  the formula for $\Lambda_{ij}$ shows that
  \begin{align*}
    \Lambda_{ij} &= y_i + z_j \ge \E[y_i] - |y_i - \E[y_i] | + \E[z_j] - |z_j - \E[z_j] | \\
    &\ge \gamma \tau - 12 \gamma \rbra{ 1 + \frac{1}{\sqrt{m}} } \max \bra{ \sqrt{\sigma_q^2 \frac{\log N}{m} }
    ,  \frac{\log N}{m} }
  \end{align*}
  for all $i \in \bar U$, $j\in \bar V$ w.h.p.;
  here, we use the assumption that $m \le n$.
  Choosing
  \begin{equation} \label{eq: tau}
    \tau = 12   \rbra{ 1 + \frac{1}{\sqrt{m}} } \max \bra{ \sqrt{\sigma_q^2 \frac{\log N}{m} }
    ,  \frac{\log N}{m} }
  \end{equation}
  ensures that the entries of $\bs \Lambda$ are nonnegative w.h.p.

  \subsection{Nonnegativity of $\bs \Xi$}

  \newcommand{\BX}{{\bs\Xi}}

  We next establish
  conditions on the regularization parameter $\gamma$
  ensuring that the entries of the dual variable $\BX$ are  nonnegative.
  Recall that $\Xi_{ij} $ takes value $0$ or $\gamma$
  for all $(i,j)$ except those corresponding to Cases 4 through 6 in the choice of $\W$ and $\BX$.

  We begin with  Case 5 in the construction of $\W$ and $\BX$.
  Recall that
  \begin{equation} \label{eq: xi5}
    \Xi_{ij} = \gamma - \frac{\lambda m }{m-\nu_j}
      = \frac{1}{m-\nu_j} \rbra{ \gamma (m q - \nu_j ) - \gamma m \tau - \frac{m}
        {\sqrt{mn} } }
  \end{equation}
  if $i \in \bar U$ and $j \notin \bar V$ such that $ij \in \Omega$.
  Since $\nu_j$ is a binomial random variable corresponding to $m$ independent Bernoulli trials
  with probability of success $p$, applying Bernstein's inequality~\eqref{eq: Bernstein}
  shows that
  $
    \nu_j \ge pm + 6 \max \{ \sqrt{ \sigma_p^2 m\log N}, \log N\}
  $,
  and, hence,
  \begin{align*}
    \Xi_{ij} \ge \frac{m}{m-\nu_j} \rbra{ \gamma \rbra{ q - p -
    {6\max\bra{\sqrt{\sigma_p^2 \frac{\log N}{m}}, \frac{\log N}{m} } - \tau} } - \frac{1}{\sqrt{mn}} }
  \end{align*}
  w.h.p., where $\sigma_p^2:= p(1-p)$.
  Under the gap assumption
  \begin{equation} \label{eq: gap}
    q - p \ge 18 \rbra{1 + \frac{1}{\sqrt{m}} }\max \bra{ \max \{ \sigma_q, \sigma_p \} \sqrt{\frac{\log N}{m}} , \frac{\log N}{m} }
  \end{equation}
  and the choice of $\tau$ given by
  we see that
  \[
    \Xi_{ij}  \ge \frac{m}{m- \nu_j} \rbra{
    \frac{\gamma}{\tilde c}(q-p) - \frac{1}
    {\sqrt{mn}} }
  \]
  w.h.p.~for some constant $\tilde c \ge 3$. An identical bound holds for entries of $\BX$ corresponding to Case 6 by symmetry.
  Finally, the bound for Case 4 follows by substituting $\nu_i = pm$ in~\eqref{eq: xi5} which establishes that $\Xi_{ij} \ge 0$ if $\gamma(q-p) \ge 3/\sqrt{mn}$ in this case. Applying the union bound over
  all entries in $\BX$  establishes that $\BX$ is nonnegative
  w.h.p.~if $q$ and $p$ satisfy the gap assumption~\eqref{eq: gap} and
  \begin{equation} \label{eq: gamma}
    \gamma \ge \frac{\tilde c}{(q-p) \sqrt{mn}}.
  \end{equation}

  \subsection{A bound on the matrix $\W$}

  \newcommand{\Q}{{\bs Q}}
  \renewcommand{\H}{\bs{H}}
  \renewcommand{\S}{\bs{S}}

  To complete the proof, we derive a sufficient condition involving $m, n, M, N ,p$, and $q$ that ensures that $\W$, as constructed above,
  satisfies $\|\W \| < 1$ with high probability.
  To simplify our notation, we again make the assumption that $m \le n$ and $M \le N$.
  Our analysis will translate superficially to the cases when $m\le n$ and $M \ge N$,
  $ m \ge n$ and $M \le N$, and $m\ge n$ and $M \ge N$.
  We bound $\|\W \|$ using the triangle inequality and the decomposition $\W = \gamma \Q +
  \lambda \S$, where
  $\gamma Q_{ij} = W_{ij}$ if $i \in \bar U$, $j\in \bar V$ and $\gamma Q_{ij} = 0$ otherwise.
  To bound the norms of $\Q$ and $\S$, we will make repeated use of the following bound on the norm of a random matrix.
  Specifically, Lemma~\ref{lem: Tropp} is a special case of the matrix concentration inequality given by~\cite[Corollary 3.11]{bandeira2014sharp} on the spectral norm of
  matrices with i.i.d.~mean zero bounded entries.

  \begin{lemma}\label{lem: Tropp}
    Let $\bs A = [a_{ij}]\in \R^{m\times n}$ be a random  matrix with i.i.d.~mean zero entries
    $a_{ij}$ having
    variance $\sigma^2$ and satisfying $|a_{ij}| \le B$.
    Let $n_{\max} = \max\{m,n\}$. Then there is a constant $c >0$ such that
    \begin{equation} \label{eq: Tropp}
      \Pr\rbra{\|\bs A\| > c \max \bra{\sqrt{\sigma^2 n_{\max}},
      \sqrt{B \log  t} }  } \le n_{\max} t^{-7}
    \end{equation}
    for all $t > 0$.
  \end{lemma}

  %
  %


  The following lemma provides the desired bound on $\|\Q\|$.

  \begin{lemma}\label{lem: Q}
    Suppose that the matrix $\W$ is constructed according to Cases 1 through 6
    for a matrix $\A$ sampled from the planted dense submatrix
    model with $m \le n$ and $M \le N$. Then there is a constant $C_Q > 0$  such that
    \begin{equation*} 
       \|\Q \|
       \le C_Q  \max
      \bra{ \sqrt{\sigma_q^2 n \log N} , \sqrt{\frac{n}{m} } \log N }
    \end{equation*}
    with high probability.
  \end{lemma}

  We delay the proof of Lemma~\ref{lem: Q} until Appendix~\ref{app:Q-bound}.
  The following lemma provides the required bound on $\|\S \|$.
  Our analysis follows a similar argument to that of~\cite[Section 4.2]{ames2011nuclear}; we include it here
  for completeness.

  \newcommand{\ts}{\tilde\sigma^2_p}
  \begin{lemma}\label{lem: S bound}
    Suppose that the matrix $\W$ is constructed according to Cases 1 through 6
    for a matrix $\A$ sampled from the planted dense submatrix
    model with $m \le n$ and $M \le N$.
    Let $\tilde\sigma^2_p:= p/(1-p)$ and $B := \max \{1, \ts \}$.
    Assume that $p$ is bounded away from $1$ so that $B = O(1)$.
    Then exists constant $C_S > 0$  such that
    \begin{equation*} 
       \|\S \| \le  { C_S  } \max \bra{ \sqrt{ \ts N \log N}, \log N \\ } \sqrt{\log N}
    \end{equation*}
    with high probability.
  \end{lemma}

  It follows immediately from Lemmas~\ref{lem: Q}~and~\ref{lem: S bound}
  that
  \begin{equation} \label{eq: W int}
    \|\W \| \le {C_Q}   \gamma \max\bra{ \sqrt{ \sigma_q^2 n \log N}, \sqrt{\frac m n } \log N }
      + C_S \, \lambda \, \max \bra{ \sqrt{\ts N} \log N, (\log N )^{3/2}}
  \end{equation}
  w.h.p.
  On the other hand,
  \begin{align*}
    \lambda &= \frac{1}{\sqrt{mn}} + \gamma (1-q) + \gamma \tau 
    \le \frac{1}{\sqrt{mn}} \rbra{1 + \frac{\tilde c}{(q-p) \sqrt mn} 1 - q + \frac{1}{2}(q-p)}  
    \le \frac{\tilde c}{(q-p)\sqrt{mn} } (1-p),
  \end{align*}
  where we obtain the first inequality by substituting the choice of $\gamma$ given by~\eqref{eq: gamma}
  and the upper bound  $\tau \le 3(q-p)/2 \le \tilde c (q-p)/2$.
  We obtain the last inequality using  the fact that $(1/\tilde c + 1/2)(q-p) \le q-p$.
  Further, we have
  $$
    \lambda \tilde\sigma_p \le \frac{\tilde{c}  (1-p) }{(q-p)\sqrt{mn} } \sqrt{ \frac{p}{1-p} }
    = \frac{\tilde c \sqrt{ p (1-p)}}{(q-p)\sqrt{mn} } = \frac{\tilde c \sigma_p}{(q-p)\sqrt{mn} } .
  $$
  Substituting back into~\eqref{eq: W int}, we see that
  \begin{align}
    \|\W \| &= O \rbra{ \frac{1}{(q-p)\sqrt{mn} } \rbra{ \max\bra{\sqrt{\sigma_q^2 n \log N}, \sqrt{\frac{n}{m}}\log N }
        + \max \bra{\sqrt{\sigma_p^2 N} \log N, (\log N )^{3/2}  } }}\notag \\
        &= O \rbra{ \frac{1}{q-p} \max \bra{ \sqrt{\frac{ \sigma_q^2 \log N}{m}},
          \sqrt{ \frac {\sigma_p^2 N }{mn}} \log N,
          \frac{(\log N )^{3/2}}{m} } } \label{eq: W bound}
  \end{align}
  w.h.p. Enforcing $q-p$ so that~\eqref{eq: gap} holds and the right-hand side of~\eqref{eq: W bound}
  is bounded above by $1$ establishes Theorem~\ref{thm: dsm rec}. This completes the proof.
  \qed

\section{A First-Order Method Based on the Alternating Direction Method of Multipliers}
\label{sec:expts}

We conclude with discussion of an optimization algorithm for solution of~\eqref{eq: dsm rel} based on the alternating direction method of multipliers (ADMM); see~\cite{boyd2011distributed} for details regarding the ADMM. We first present a derivation of the method and then empirically validate its performance using randomly generated matrices and real-world collaboration and communication networks.

\subsection{The Optimization Algorithm}
\newcommand{\Z}{\bs{Z}}
To apply the ADMM to~\eqref{eq: dks rel}, we first introduce artificial variables $\Q$, $\W$, $\Z$ to obtain the equivalent convex optimization problem
\begin{equation}\label{eq:split-prob}
  \begin{array}{ll}
    \min & \|\X\|_*+\bs{\gamma}\|\Y\|_1+\mathbbm{1} _{\Omega_Q}(\bs{Q})+\mathbbm{1}_{\Omega_W}(\W)+\mathbbm{1}_{\Omega_Z}(\bs{Z}) \\
    & \X = \Y = \Q,\; \X - \W = \0,\; \X - \Z = \0,
  \end{array}
\end{equation}
where $\Omega_Q,  \Omega_W,  \Omega_Z$ denote the constraint sets
\[
  \Omega_Q := \{\bs{Q} : P_{\tilde{N}}(\bs{Q})=\0 \}, \hspace{0.25in}
  \Omega_W :=\{\W:  \e^T\W\e=mn \, \}, \hspace{0.25in}
  \Omega_Z =\{\bs{Z}: {Z}_{ij}\leq 1  \; \forall (i,j)\in M\times N \, \}.
\]
Here, $\mathbbm{1}_{\S}: \R^{M\times M} \rightarrow \left \{0,+\infty \right \}$  is the indicator function of the set $S \subseteq  \R^{M\times N}$,
such that
$\mathbbm{1}_{S}(\X)=0$  if $\X\in S$, and $+\infty$ otherwise.
We solve~\eqref{eq:split-prob} iteratively using the ADMM.
Specifically, we update each primal variable by minimizing the augmented Lagrangian in Gauss-Seidel fashion with respect to each primal variable. Then the dual variables are updated using the updated primal variables.
The augmented Lagrangian of~\eqref{eq:split-prob} is given by
\begin{align*}
  L_{\tau}=\|\X\|_*+&\bs{\gamma}\|\Y\|_1 +\mathbbm{1}_{\Omega_Q}(\bs{Q})
  +\mathbbm{1} _{\Omega_W}(\W)
  + \mathbbm{1}_{\Omega_Z}(\bs{Z})
  +\tr(\bs{\Lambda_Q}(\X-\Y-\Q))
  +\tr(\bs{\Lambda_W}(\X-\W))\\
  &+\tr(\bs{\Lambda_Z}(\X-\bs{Z}))
  +\frac{\tau}{2} \rbra{\|\X-\Y-\Q\|_F^2 +\|\X-\W\|_F^2+\|\X-\bs{Z}\|_F^2},
\end{align*}
where
$\tau$ is a regularization parameter chosen so that $L_\tau$ is  strongly convex in each primal variable. Minimization of the augmented Lagrangian with respect to each of the artificial primal variables $\bs{Q},\W$ and $\bs{Z}$ is equvalent to projection onto each of the sets $\Omega_Q, \Omega_W$ and $\Omega_Z$; each of these projections has an analytic expression.
%
%
%
%
\newcommand{\M}{\bs{M}}
We update $\Y$ using projection onto the nonnegative cone: $P_{\R^{M\times N}_+}(\M)$ is the matrix with $ij$th entry $m_{ij}$ if $m_{ij}\ge0$ and $0$ otherwise.
On the other hand, we update $\X$ using the proximal function for the nuclear norm $\|\cdot\|_*$, which can be computed by applying the soft thresholding operator defined by
$
  S_\phi(\x) = \textrm{sign}(\x) \max\bra{ |\x| - \phi \e, \bs 0}
$
to the vector of singular values. Here, $\textrm{sign}(\x)$ is the vector whose entries are the signs of the corresponding entries of $\x$, $|\x|$ denotes the vector whose entries are the magnitudes of the corresponding entries of $\x$, and the maximum denotes the vector of pairwise maximums.
We declare the algorithm to have converged when the primal and dual residuals
$\|\X^{l+1}-\W^{l+1}\|_F, \|\X^{l+1}-\bs{Z}^{l+1}\|_F,\|\W^{l+1}-\W^{l}\|_F, \|\bs{Z}^{l+1}-\bs{Z}^{l}\|_F$, and $\|\bs{Q}^{l+1}-\bs{Q}^{l}\|_F$
are smaller than a desired error tolerance. The steps of the algorithm are summarized in Algorithm~\ref{admm}\footnote{A MATLAB implementation of Algorithm~\ref{admm} is available from \href{http://bpames.people.ua.edu/software}{http://bpames.people.ua.edu/software}.}.

\begin{algorithm}[t]
\begin{algorithmic}
\State\textbf{Inputs}
Binary matrix $\A= \R^{M \times N}$, $m\in [M]$, $ n\in [N]$, regularization parameters $\gamma$, $\tau$, and stopping tolerance $\epsilon$
\State\textbf {Initialize}
$\mu=1/\tau$, $\X^{0}=\W^{0}=\Y^{0}=(mn/MN)\e\e^{T}$
\For {$l=0,1,...$ until converged}
\State\textbf{Step 1.}
Update  $\bs{Q}^{l+1}$:
 $\bs{Q}^{l+1}=P_{\Omega} \rbra{\X^{l}-\Y^{l}+\mu \bs{\Lambda_{Q}}^{l}}$
\State\textbf{Step 2.}
Update $\X^{l+1}$: Take SVD  of
\[ \tilde{\X}^{l}=1/3\Big(\Y^{l}+\bs{Q}^{l+1}+\bs{Z}^{l}+\W^{l}-\mu(\bs{\Lambda_Q}^{l}+\bs{\Lambda_Z}^{l}+\bs{\Lambda_W}^{l})\Big)=\bs{U}(\Diag \bs{x})\bs{V}^T.\]
\hspace{0.75in} Apply soft thresholding: $\X^{l+1}=\bs{U} (\Diag S_{\tau/3} (\bs{x}))\bs{V}^T$.
\State\textbf{Step 3.}
Update $\Y^{l+1}$:  $y^{l+1}_{ij} = \max\bra{\sbra{\X^{l+1}-\bs{Q}^{l+1}-\gamma \e \e^T\mu + \bs{\Lambda_Q}^{l}\mu}_{ij}, 0 }$ for all  $i\in [M]$, $j\in[N]$.
\State\textbf{Step 4.}
Update $\W^{l+1}$: $\W^{l+1}=\tilde{\W}^{l}+\alpha^{l}\e\e^{T}$, where $\tilde{\W}^{l}=\X^{l+1}+\mu\bs{\Lambda_W}^{l}$ and $\alpha^{l}=(mn-\e^T\tilde{\W}^{l}\e)/(MN)$.
 \State\textbf{Step 5.}
 Update variable $\bs{Z}^{l+1}:$
 $z_{ij}^{l+1}=\min \{\max \bra{ \sbra{\X^{l+1}+\mu\bs{\Lambda_Z}^{l}}_{ij},0\},1 }$ for all $i\in [M]$, $j\in[N]$.
 \State\textbf{Step 6.}
 Update dual variables:
 \begin{align*}
 \bs{\Lambda_Q}^{l+1} &=\bs{\Lambda_Q}^{l}+\tau(\X^{l+1}-\Y^{l+1}-\bs{Q}^{l+1}), \\
\bs{\Lambda_W}^{l+1} &=\bs{\Lambda_W}^{l}+\tau(\X^{l+1}-\W^{l+1}), \\
\bs{\Lambda_Z}^{l+1} &=\bs{\Lambda_Z}^{l}+\tau(\X^{l+1}-\bs{Z}^{l+1}).
 \end{align*}
 \State\textbf{Step 7.}
 Calculate primal and dual residuals
 \begin{align*}
r_p &=\max \{ \|\X^{l+1}-\bs{Z}^{l+1}\|_F,\|\X^{l+1}-\W^{l+1}\|_F, \|\X^{l+1}-\Y^{l+1}-\bs{Q}^{l+1}\|_F\}/\|\X^{l+1}\|_F,\\
r_d &=\max \{\|\bs{Z}^{l+1}-\bs{Z}^{l}\|_F, \|\W^{l+1}-\W^{l}\|_F, \|\bs{Q}^{l+1}-\bs{Q}^{l}\|_F\} /\|\X^{l+1}\|_F.
\end{align*}
\If {$\max(r_p,r_d)<\epsilon$} \text {algorithm converged.}
\EndIf
\EndFor
\end{algorithmic}
\caption{ADMM for solving relaxation of densest $(m,n)$-submatrix problem~\eqref{eq: dsm rel}}
\label{admm}
\end{algorithm}


%
%
%
%

\subsection{Random matrices}
We empirically verified the theoretical phase transitions provided by Theorem~\ref{thm: dsm rec} using matrices randomly sampled from the planted dense subgraph model with fixed noise edge probability $p$ and varied the submatrix size $n$ and in-submatrix probability $q$.
We perform two sets of experiments: one where the matrix is sparse outside the planted submatrix and another when the noise obscuring the planted submatrix is relatively dense.
For the dense graph simulations, we choose $p=0.25$ and $q \in \{0.25, 0.30, \dots, 0.95, 1\}$.
In the sparse experiments, we choose $p = 1/\sqrt{N}$ and $q = t p$ for ten equally spaced $t$ spanning the interval $[2,\sqrt{N}]$.
For each set of simulations, we vary $n \in \{10, 20, 30, \dots, 240,250\}$ and set $m = 2n$.
In the sparse experiments, we have $M = N = 1000$ and we use $M = N = 500$ in dense experiments.
In both the dense and sparse graph simulations, we generate $10$ matrices according to the planted dense submatrix model for each choice of the parameters $q$ and $n$ (with remaining parameters $p$, $M$, and $N$ chosen as described above).
 We call Algorithm~\ref{admm} to solve the instance of~\eqref{eq: dsm rel} corresponding to each randomly sampled matrix.
The regularization parameter $\gamma=6/(q-p)n$, augmented Lagrangian parameter $\tau=0.35$, and stopping tolerance $\epsilon=10^{-4}$ are used in each call of Algorithm~\ref{admm}. We declared the planted submatrix to be recovered if the relative error $\|\X^* - \X^0\|_F /\|\X^0\|_F$ is less than $10^{-3}$, where $\X^*$ is the solution returned by Algorithm~\ref{admm} and $\X^0$ is the matrix representation of the planted submatrix.
The empirical probability of recovery of planted submatrix  is plotted in  Figure~\ref{fig:recovery}. Color of a square indicates  rate of recovery in the corresponding simulations, with black corresponding to $0$ and white corresponding to $10$ recoveries out of $10$ trials. The dashed curves show the phase transition to perfect recovery predicted by Theorem~\ref{thm: dsm rec}.
The empirical recovery rates observed in these trials closely matches that predicted by Theorem~\ref{thm: dsm rec}. The discrepancy between the observed phased transition and that more conservatively predicted by Theorem~\ref{thm: dsm rec} is due to the presence of the logarithmic terms in~\eqref{eq: gap 1}; a slight modification of our proof to follow that of \cite[Theorem~7]{ames2011nuclear} eliminates these terms when $p$ and $q$ are constants and the gap $q-p$ is sufficiently large.

\begin{figure}[t]
	\centering
  \subfloat[Dense case: $A=500 \times 500$, $p=0.25$ ] {%
		\includegraphics[height=2in]{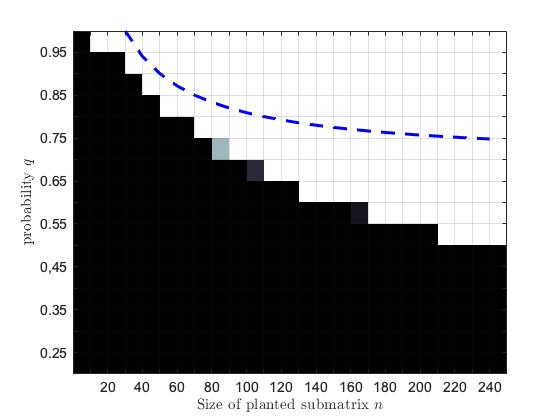}}%
		\qquad
	\subfloat[Sparse case: $A=1000\times 1000$, $p=1/\sqrt{N}$]{%
	\label{fig:sparse}
		\includegraphics[height=2in]{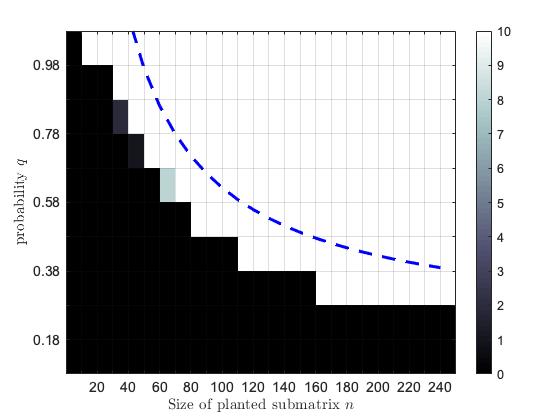}}%
	\caption{Recovery rates for randomly generated matrices.}\label{fig:recovery}
\end{figure}
\newpage

\subsubsection{Collaboration and Communication Networks}
We also applied our algorithm to identify communities in networks taken from the 10th DIMACs Implementation Challenge, which focused on graph partitioning and clustering~\cite{bader2012graph,bader2014benchmarking}.
The first graph (JAZZ)  represents a collaboration network with $198$ musicians and $2742$ edges, and was compiled by ~\cite{gleiser2003jazz}. Here, two musicians are connected if they have performed together. Earlier studies~\cite{tsourakakis2013denser} showed that this network contains a cluster of $100$ musicians. We apply Algorithm~\ref{admm} to the adjacency matrix of this network with  regularization parameter $\tau=0.85$, stopping tolerance $\epsilon = 10^{-2}$, and $m=n = 100$. Our algorithm converges to the dense submatrix representing this community after $50$ iterations.
Figure~\ref{fig:JAZZ} is a visualization of this network using the software package Gephi~\cite{bastian2009gephi} and the ForceAtlas2 algorithm~\cite{jacomy2014forceatlas2}. The \texttt{statistics} function of Gephi is used to identify three communities within this network, including the community of size $n=100$ identified by Algorithm~\ref{admm}.

\begin{figure}[t]
\graphicspath{ {/home/user/research} }
	\centering
	\subfloat[JAZZ Network] {%
		\label{fig:jazz} %
		\includegraphics[height=2 in]{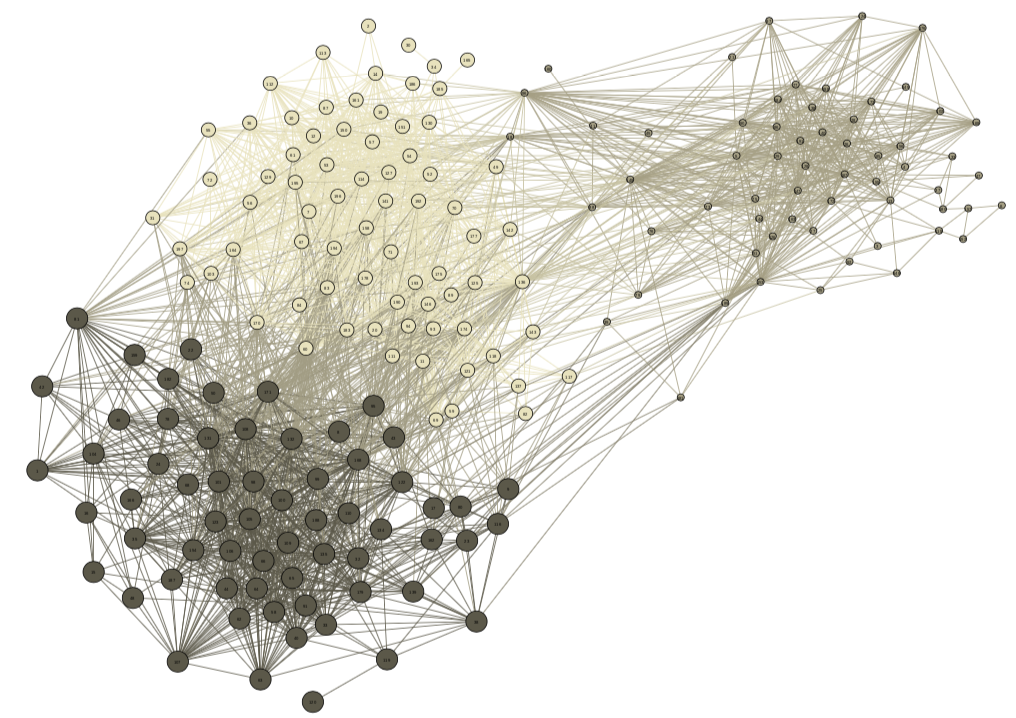}}
		\qquad
	\subfloat[Extracted communities]{%
	\label{fig:communities}
		\includegraphics[height=2 in]{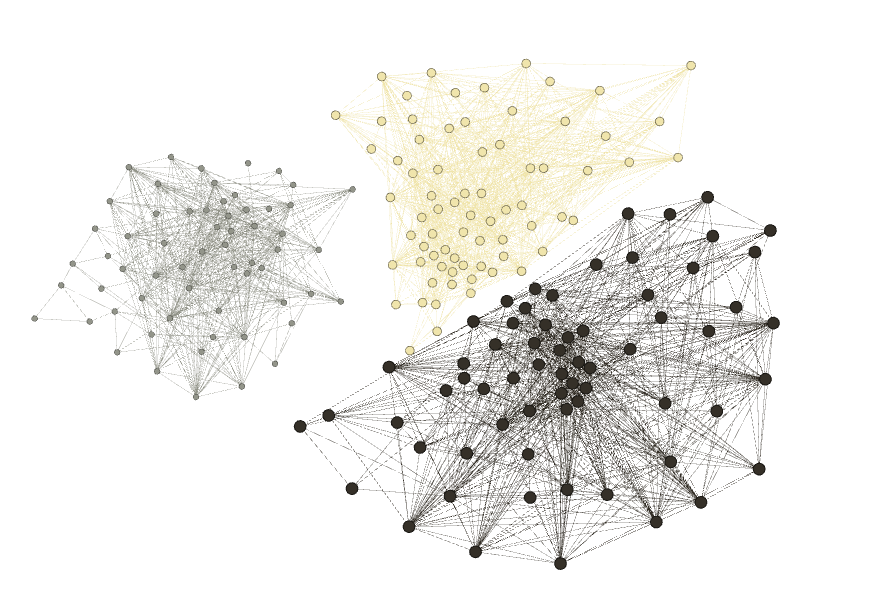}}%
			\caption{JAZZ Network. Each color corresponds to membership in one of 3 clusters, isolated in the plot on the right.}
      \label{fig:JAZZ}
\end{figure}

We also consider the graph (EMAIL) representing the network of e-mail interchanges between faculty, researchers, technicians, managers, administrators, and graduate students of the Univeristy Rovira i Virgili (Tarragona). Two individuals are connected if they exchanged an email. There are $1133$ nodes and $5451$ edges. From~\cite{Tsourakakis2013}, we know that the EMAIL graph has a dense subgraph of $289$ vertices, representing a community of $289$; additionally, we can identify $7$ clusters using the \texttt{statistics} function of Gephi corresponding to academic units within this university, including this community. Applying Algorithm~\ref{admm} with $m=n= 289$, $\tau=0.35$, and stopping criteria $\epsilon=10^{-2}$ finds this subgraph in $15$ iterations. 
The results of these analyses are summarized in Table~\ref{table:real}.

\begin{table}[t!]
\centering
\begin{tabular}{c c c c c} 
\toprule
Graph & Number of vertices & Number of edges & Size of dense subgraph  & Running time \\\midrule
JAZZ & 198 & 2742 & 100 & 0.605735 sec\\ 
EMAIL & 1133 & 5451 & 289 &  20.139186 sec \\ [1ex] 
\hline 
\end{tabular}
\caption{Densest subgraphs extracted with ADMM.}
\label{table:real}
\end{table}

\section{Conclusions}

We have presented an analysis of new convex relaxations for the densest subgraph and submatrix problems and have  established sufficient conditions under which the optimal solution of original combinatorial problem coincides with that of these relaxations. In particular, these sufficient conditions characterize a signal-to-noise (SNR) ratio for matrices sampled from a particular distribution of random matrices, such that if this ratio is sufficiently large then we can expect to recover the combinatorial solution from the solution of the relaxation. Here, we expect perfect recovery if the strength of signal, as measured by the gap between probabilities of existence of nonzero within-group edges and out-group entries, is sufficiently larger than noise, as measured by variability of presence of nonzero entries. Further, the SNR corresponding to this phase transition to perfect recovery matches the current state of the art identified in the previous literature (up to constant and logarithmic terms); see~\cite{brennan2019universality}.

This recovery  guarantee provides a \emph{sufficient} condition for perfect recovery of the planted subgraph or submatrix. It would be very interesting to determine if this condition is also \emph{necessary}. For example, we establish that we have perfect recovery of planted dense $k$-submatrices and subgraphs if $k \ge \Omega(N)$ when the probabilities $q$, $p$ are sufficiently large constants. It is unclear if it is possible, either using our relaxation or some other method, to efficiently recover planted submatrices and subgraphs of size $O(n^{1/2 - \epsilon})$ for some $\epsilon >0$.

A secondary open problem focuses on efficient solution of the proposed convex relaxations. We currently solve these problems using a multi-block variant of the ADMM. Each iteration of this algorithm requires $O(N^3)$ arithmetic operations; the bulk of these operations are used by the calculation of the singular value decomposition used to update $\X$. This per-iteration cost scales unfavorably when $N$ is large.  The recent manuscript by Sotirov~\cite{sotirov2019solving} proposed a coordinate descent heuristic for the densest $k$-subgraph, with empirical evidence that this heuristic efficiently solves large-scale instances of the densest $k$-subgraph problem. However, sufficient conditions for perfect recovery of a planted dense submatrix have not yet been established for this method has not been. Further research is needed to design efficient and scalable algorithms, i.e.,  with per-iteration cost $O(N)$, with provable theoretical guarantees of recovery for the solution of the densest subgraph and submatrix problems.

\noindent \textbf{Acknowledgements.}
B.~Ames was supported by University of Alabama Research Grants RG14678 and RG14838.

\appendix

\section{Proof of Lemma~\ref{lem: Q}}
\label{app:Q-bound}

The proof is virtually identical to that of~\cite[Lemma~4.5]{ames2015guaranteed}.
We decompose $\Q$ as $\Q = \lambda \e\e^T- \bs H + \y\e^T + \e\z^T$,
where $\bs H$ is matrix defined by
$H_{ij} = 1$ if $ij \in \Omega$ and $H_{ij} = 0$ otherwise.
We can further decompose $\Q$ as
$
  \Q = \Q_1 + \Q_2 + \Q_3 + \Q_4,
$
where $\Q_1, \Q_2, \Q_3, \Q_4$ are constructed as  below.

We first bound $\Q_1 := (1-q) \e\e^T - \bs H$.
Note that $\Q_1$ has i.i.d.~mean-zero entries, with variance
$\sigma^2 = \sigma_q^2$ and values either $1-q$ with probability $q$
or $-q$ with probability $1-q$.
Applying~\eqref{eq: Tropp} with $B = 1$ and $t = N$,
\begin{align}
  \|\Q_1\| = O \rbra{ \max\bra{\sqrt{\sigma_q^2 n }, \sqrt{\log N} }}\label{eq: Q1}
\end{align}
w.h.p.
Next, we let $\Q_2 := \frac{1}{n}(\bm \e^T - (1-q) n\e\e^T )$.
Note that
\begin{align*}
  \|\Q_2 \|_2 &= \frac{1}{n} \| (\bm - (1-q)n \e) \e^T \|_2 
  \le \frac{1}{n} \| \bm - (1-q)n \e \| \|\e \| 
   = \frac{1}{\sqrt{n}} \| \bm - (1-q) n \e \|. 
\end{align*}
Applying~\eqref{eq: mi} shows that
$
  \bar\mu_i - (1-q) n \le 6 \max \bra{ \sqrt{\sigma_q^2 n\log N}, \log N }
$
for all $i \in \bar U$
w.h.p. It follows that
\begin{equation} \label{eq: Q2}
  \|\Q_2\|  \le 6 \sqrt{ \frac{m}{n} }\max \bra{ \sqrt{\sigma_q^2 n\log N}, \log N }
\end{equation}
w.h.p.
Next, let $Q_3 := \frac{1}{m} \e \bn^T - (1-q) \e\e^T$.
An identical argument shows that
\begin{equation} \label{eq: Q3}
  \|\Q_3\| \le 6 \sqrt{\frac{n}{m}} \max \bra{ \sqrt{\sigma_q^2 m\log N}, \log N }
\end{equation}
w.h.p.
Finally, we let
\[
  \Q_4 := \rbra{ \frac{(1-q) mn - \bm^T \e}{mn} } \e\e^T.
\]
It is easy to confirm that $\gamma ( \Q_1 + \Q_2 + \Q_3 + \Q_4) = \W(\bar U, \bar V)$.
Applying~\eqref{eq: m sum} shows that
\begin{equation} \label{eq: Q4}
  \|\Q_4 \| \le \frac{1}{\sqrt{mn}} 6 \max \bra{ \sqrt{\sigma_q^2 mn \log N }, \log N }
\end{equation}
w.h.p.
Combining~\eqref{eq: Q1},~\eqref{eq: Q2},~\eqref{eq: Q3}, and~\eqref{eq: Q4}
establishes that
\[
  \|\Q \| \le \sum_{i=1}^4 \|\Q_i \| = O\rbra{  \max
  \bra{ \sqrt{\sigma_q^2 n \log N}, \sqrt{\frac{n}{m} } \log N } }
\]
w.h.p., as required.
\qed

\section{Proof of Lemma~\ref{lem: S bound}}
\label{app:S-bound}

To obtain the desired bound on $\S$, we first approximate $\S$ with a random matrix with mean zero entries.
In particular, we let $\tilde \S_1$ be the random matrix constructed as follows.
For all $(i,j) \notin \bar U \times \bar V$ such that $ij \notin \Omega$, or $(i,j) \in ([M]- \bar U)\times ([N]-\bar V)$ such that $(i,j) \in \Omega$,
we let ${[\tilde \S_1]}_{ij} = S_{ij}$.
All remaining entries of $\tilde \S_1$ are sampled independently from the generalized Bernoulli distribution $\mathcal{B}$, where $x$ sampled from $\mathcal{B}$
satisfy
$$
  x = \branchdef{\lambda , &\mbox{with probability } p, \\
  -\lambda \, \ts, & \mbox{with probability } 1-p . }
$$
Note that $\tilde \S_1$ is a random matrix with i.i.d.~mean zero entries
sampled independently from $\mathcal{B}$ by our choice of $\W$.
Applying Lemma~\ref{lem: Tropp} shows that
\begin{equation} \label{eq: S1}
  \|\tilde \S_1 \| = O\rbra{  \max \bra{ \sqrt{ \ts N}, \sqrt{B \log N} } }
\end{equation}
w.h.p.
The remainder of the proof establishes that $\S$ is well-approximated by $\tilde \S_1$, i.e., we complete the proof by bounding  the norm of the error $\S - \tilde\S_1$.
We begin with the error in the $\bar U \times \bar V$ block.
Let $\tilde \S_2 = -\tilde \S_1(\bar U, \bar V)$. Applying Lemma~\ref{lem: Tropp} with $t = N$ again
shows that
\begin{equation} \label{eq: S2}
  \|[\S - \tilde \S_1](\bar U, \bar V)\| = \|\tilde \S_2\|
  = O \rbra{  \max \bra{ \sqrt{ \ts n}, \sqrt{B \log N} } }\end{equation}
w.h.p.
We define  $\tilde \S_3$  by
$$
  {[\tilde \S_3]}_{ij} = \branchdef{ -\frac{ \nu_j}{m - \nu_j} + \frac{ p}{1-p},
  &\mbox{if } (i,j) \in \Omega, \; i \in \bar U, j\in [N] -\bar V, \\ 0, &\mbox{otherwise.} }
$$
To bound the norm of $\tilde \S_3$, we will use the
following lemma, which provides a bound on the spectral norm
of random matrices of this form.

\newcommand{\bern}{\mathcal{B}}

\begin{lemma}\label{lem: thm 4}
  Let $\A$ be an $n \times N$ matrix whose entries
  are chosen according to $\bern$ with $n \le N$.
  Let $\tilde\A$ be the random matrix defined by
  \[
    {[\tilde A]}_{ij} :=
    \begin{cases}
      1, & \text{if } A_{ij} = 1, \\
      \frac{-n_j }{n-n_j}, & \text{if } A_{ij} = -\ts,
    \end{cases}
  \]
  where $n_j$ is the number of $1$'s in the $j$th column of $\A$.
  Then there are constants $c_1, c_2 > 0$ such that
  \[
    \Pr \rbra{ \|\A - \tilde \A \| \ge
    c_1 \max \bra{ \sqrt{\ts N} \log N, (\log N)^{3/2} } } \le c_2 N^{-5}.
  \]
\end{lemma}

The proof of Lemma~\ref{lem: thm 4} follows a similar argument
to that of~\cite[Theorem 4]{ames2011nuclear} and is included as
Appendix~\ref{app:thm4-proof}.
It is easy to see that the nonzero block $\tilde \S_3$ has form $\A-\tilde \A$
as in the hypothesis of Lemma~\ref{lem: thm 4}.
It follows that
\begin{equation} \label{eq: S3}
  \|[\S - \bs{S_1}](\bar U, [N]- \bar V)\| = \|\tilde \S_3\|
  = O\rbra{ \max \bra{ \sqrt{\ts N} \log N, (\log N)^{3/2}}  }
\end{equation}
w.h.p.
Similarly, we define the final correction matrix by
$$
  {[\tilde \S_4]}_{ij} = \branchdef{ -\frac{\mu_j}{n - \mu_j} + \frac{p}{1-p},
  &\mbox{if } (i,j) \in \Omega, \; i \in [M] - \bar U, j\in\bar V, \\ 0, &\mbox{otherwise.} }
$$
Applying  Lemma~\ref{lem: thm 4} to the transpose of $\tilde \S_4$
we have
\begin{equation} \label{eq: S4}
  \|\tilde \S_4 \| = O\rbra{  \max \bra{ \sqrt{\ts M} \log M, (\log M)^{3/2}} }
  = O\rbra{ \max \bra{ \sqrt{\ts N} \log N, (\log N)^{3/2}}}
\end{equation}
w.h.p.
Combining~\eqref{eq: S1},~\eqref{eq: S2},~\eqref{eq: S3}, and~\eqref{eq: S4}
and applying the union bound completes the proof.
\qed

\section{Proof of Lemma~\ref{lem: thm 4}}
\label{app:thm4-proof}

The result relies on an application of the Matrix Bernstein Inequality; see~\cite[Theorem~6.1.1]{tropp2015introduction} and \cite[Theorem 1.6]{tropp2012user} for further details. We first state the necessary bound on the spectral norm of the sum of finitely many independent, bounded random matrices.

\renewcommand{\S}{\bs{S}}

\begin{theorem}[Matrix Bernstein Inequality]\label{thm: mat bern}
  Let $\{\S_k\}$ be a finite sequence of independent $d_1 \times d_2$ random matrices such that
  $\E[\S_k] = \0 $ and $\|\S_k\| \le L$ for all $k$
  almost surely.
  Let $\Z := \sum_k \S_k$ and let $v(\Z)$ denote the matrix variance defined by
  \begin{align} 
    v(\Z) & = \max \bra{ \|\E[\Z\Z^*] \|, \|\E[\Z^* \Z]\|} 
      = \max \bra{ \left\| \sum_k \E[\S_k\S_k^*] \right\|, \left\| \sum_k \E[\S_k^* \S_k] \right\| }
      \label{eq: mat var}.
  \end{align}
  Then
  \begin{equation}\label{eq: MBI}
    Pr (\|\Z \| \ge t) \le (d_1 + d_2) \exp \rbra{ \frac{ - t^2/2}{v(\Z) + Lt/3}}
  \end{equation}
  for all $t > 0$.
\end{theorem}

The remainder of the proof consists of a specialization of this inequality to the special case $\Z = \A - \tilde \A$.
Indeed, let
\begin{equation}\label{eq:Sj-def}
  \S_j = \d_j \e_j^T,
\end{equation}
where $\d_j := [\A-\tilde \A]_j$ denotes the $j$th column of $\A-\tilde \A$ and $\e_j$ denotes the $j$th standard basis vector.
It is clear that $\Z = \A - \tilde \A = \sum_{j=1}^N \S_j$. It remains to estimate  an upper bound $L$ on the spectral norms of the matrices $\{\S_j\}$ and an upper bound on the variance $v(\Z)$. Once we have estimated these quantities, we will substitute them into~\eqref{eq: MBI} to complete the proof.

We begin with the following estimate on $L$, which is an immediate consequence of the standard Bernstein Inequality~\eqref{eq: Bernstein}.

\begin{lemma}\label{lem: L}
  There is a  constant $c_1 > 0$
  such that matrices $\{\S_j\}$ defined by~\eqref{eq:Sj-def} satisfy
  \begin{equation*}
    \|\S_j\| \le L := c_1\sqrt{\max\bra{1, \ts} \log N}
  \end{equation*}
  for all $j=1,2,\dots, N$ with probability at least $ 1 - 2 N^{-5}$.
\end{lemma}

\begin{proof}
  Fix $j \in \{1,\dots, N\}$.
  Note that the
    $\|\S_j\| = \|\d_j \e_j^T\| = \|\d_j\|\|\e_j\| = \|\d_j\|$.
  Moreover, the Bernstein Inequality~\eqref{eq: Bernstein}
  implies that
  \begin{align*}
    \|\d_j\|^2 &= \frac{(n_j - pn)^2}{(1-p)^2(n-n_j)}
     \le \frac{36 \max \{p (1-p) n \log N, \log^2 N\}}{(1-p)^2(n - 6 \max\{\sqrt{p(1-p) n \log N}, \log N\})} 
     = O \rbra{ \max\{1, \ts \} \log N }
  \end{align*}
  with probability at least $1 - 2 N^{-6}$,
  where the last inequality follows from the fact that $n - n_j = O(n)$ w.h.p.~(by \eqref{eq: Bernstein}) and $\log N = O(n)$ (by the gap assumption).
  Taking the square root completes the proof. \qed
\end{proof}

We next bound the matrix variance $v(\Z)$.

\begin{lemma}\label{lem:mat-var-bnd}
  The matrix $\Z = \sum_{j} \S_j$ defined by~\eqref{eq:Sj-def} satisfies
    $v(\Z) \le c \ts N$
  for any constant $c > 0$ satisfying $n - n_j > (1/c) n$ for all $j$.
\end{lemma}

\begin{proof}
  It suffices to construct upper bounds on each of $\|\E[\Z\Z^T]\|$ and $\|\E[\Z^T \Z]\|$.
  We begin with the latter.
  Note that $\S_j^T \S_j^T = (\d_j^T \d_j) \e_j \e_j^T$. This implies that $\Z^T \Z$ is a diagonal matrix with $j$th diagonal entry equal to $\|\d_j\|^2$.
  It follows that
  $
    \|\E (\Z^T\Z)\| = \max_j \E[ \|\d_j\|^2].
  $
  For each $j=1,2,\dots, N$, we have
  \begin{equation}\label{eq:E-d-bnd}
    \E[\|\d_j\|^2] = \E \sbra{ \frac{(n_j-pn)^2}{(1-p)^2(n-n_j)}}
    \le \frac{c}{(1-p)^2 n} \E[(n_j - pn)^2] = \frac{c p (1-p)n}{(1-p)^2 n} = c \ts,
  \end{equation}
  where the inequality follows from the assumption that $n - n_j \ge (1/c) n$ and the second to last inequality follows from the fact that $\E[(n_j - pn)^2]$ is equal to the variance of the binomial variable $n_j$. This implies that
  \begin{equation}\label{eq:ZTZbnd}
    \|\E[\Z^T\Z]\| \le c \ts.
  \end{equation}

  On the other hand,
  $  \E[\Z\Z^T] = \sum_{j=1}^N \E[\S_j \S_j^T] = \sum_{j=1}^N \E[\d_j \d_j^T]$.
  It follows immediately that
  \begin{equation}\label{eq:out-tri}
    \|\E[\Z\Z^T]\| \le \sum_{j=1}^N \|\E[\d_j\d_j^T]\| \le \sum_{j=1}^N \E [ \|\d_j \d_j^T\|]
    = \sum_{j=1}^N \E[ \|\d_j\|^2]
  \end{equation}
  by the triangle inequality and Jensen's inequality.
  Applying~\eqref{eq:E-d-bnd}, we see that
  \begin{equation}\label{eq:ZZTbnd}
    \|\E[\Z\Z^T]\| \le c N \ts.
  \end{equation}
  Substituting~\eqref{eq:ZTZbnd} and~\eqref{eq:ZZTbnd} into the formula for $v(\Z)$, we see that we have $v(\Z) \le c N \ts$. \qed
\end{proof}

We are now ready to complete the proof of~Lemma~\ref{lem: thm 4}. Let's consider the following cases.
First, suppose that $\ts N \ge \log N$ and let $t = \tilde c \sqrt{\ts N} \log N$ in~\eqref{eq: MBI}. Recall that applying the Bernstein Inequality~\eqref{eq: Bernstein} to each binomial variable $n_j$ implies that there is a  constant $c$ such that $n-n_j \ge (1/c) n$ for all $j$ with probability at least $1 - 2 N^{-5}$.
This implies that we have $v(\Z) \le c \ts N$ with the same probability.
Substituting into~\eqref{eq: MBI}, along with the choice of $L$ from Lemma~\ref{lem: L}, we see that
\begin{align*}
  Pr(\|\Z\| \ge t)
    &\le (N + n) \exp \rbra{ - \frac{ \tilde c^2 \ts N \log^2 N / 2}{c \ts N + \tilde c\sqrt{B} \ts N \log N/3 } }
\end{align*}
using the assumption that $\sqrt{\log N} \le \sqrt{\ts N}$.
Rearranging further, we see that
\begin{align*}
  Pr(\|\Z\| \ge t)
    &\le (N+n) \exp\rbra{ - \frac{\tilde c^2}{c + \tilde c \sqrt{B}} \log N}
    \le (N+n) \exp(-7 \log N) \le 2 N^{-6},
\end{align*}
if we choose $\tilde c$ large enough that $\tilde c^2/(c+\tilde c \sqrt{B}) > 7$ (which is possible if we impose the assumption that $B = O(1)$).

Next, consider the case that $\log N > \ts N$ and let $t = \tilde c \log^{3/2}(N)$. Then the Matrix Bernstein Inequality~\eqref{eq: MBI} implies that
\begin{align*}
  Pr(\|\Z\| \ge t)
  &\le (N + n) \exp \rbra{ \frac {-\tilde c \log^3 N /2 } {c \ts N + \tilde c \sqrt{B} \log^2 N /3} } 
  \le (N+n) \exp (-7\log N) \le 2 N^{-6}
\end{align*}
for any $\tilde c$ satisfying $\tilde c^2/(c+\tilde c \sqrt{B}) > 7$.
Combining the two cases, we see that there are constants $c_1, c_2 >0$ such that
\[
  \|\Z\| = \|\A - \tilde \A\| \le c_1 \max \bra{ \sqrt{\ts N \log N}, \log N} \sqrt{\log N}
\]
with probability at least $1 - c_2 N^{-5}$.
This completes the proof. \qed

\bibliographystyle{spmpsci_unsrt}
\bibliography{bib2}

\end{document}